\documentclass[a4paper,12pt]{amsart}
\usepackage{amsthm,euscript,epsf}
\usepackage{amsmath, amssymb, latexsym}
\usepackage{amsopn,amsfonts, amsthm}
\usepackage[T1]{fontenc}
\usepackage[english]{babel}
\usepackage{courier}% Times Roman font
\usepackage{mathrsfs}

\newcommand{\ind}{\operatorname{ind}}

\def\sub{\subset}
\def\sm{\setminus}
\def\emp{\varnothing}
\def\res{\restriction}
\let\mathcal\mathscr

\newtheorem{theorem}{Theorem}[section]

\newtheorem{corollary}[theorem]{Corollary}

\newtheorem{lemma}[theorem]{Lemma}
\theoremstyle{definition}

\begin{document}

{

\renewcommand*{\thefootnote}{$\star$}

\title{The completeness\\ of free Boolean topological groups$^\star$}
\footnotetext[0]{The work of the first listed author was financially supported by the Russian Science Foundation, 
grant~22-11-00075-P. The second listed author was partially supported by CONAHCyT 
(Ciencia de Frontera 2019) of Mexico, grant FORDECYT-PRONACES/64356/2020.} 

}

\author{Ol'ga V. Sipacheva}

\address{Department of General Topology and Geometry, Faculty of Mechanics
and Mathematics, M. V. Lomonosov Moscow State University, Leninskie Gory 1, 
Moscow, 199991 Russia}
\email[O. Sipacheva]{osipa@gmail.com}

\author{Mikhail G. Tkachenko}

\address{Department of Mathematics, Universidad Aut\'onoma Metropolitana, 
Av. San Rafael Atlixco 186, Col. Vicentina, Iztapalapa, C.P. 09340, 
Mexico City, Mexico}
\email[M. Tkachenko]{mich@xanum.uam.mx}

\begin{abstract}
It is proved that the free Boolean topological group $B(X)$ on a
Tychonoff space $X$ is Weil complete if and only if the space $X$ 
is Dieudonn\'e complete. This result provides a positive answer 
to a question posed by the first listed author in 2015. 
\end{abstract}



 \maketitle

%%%%%%%%%%%%%%%%%%%%%%%%%%%%%%%%%%%%%%%%%
%%%%%%%%%%%%%%%%  I N T R O D U C T I O N  %%%%%%%%%%%%
\section{Introduction}
%%%%%%%%%%%%%%%%%%%%%%%%%%%%%%%%%%%%%%%%%
In this article, it is assumed that all topological groups are Hausdorff 
and all spaces are Tychonoff. 

Given a space $X$, $A(X)$ and $B(X)$ denote, respectively, the free Abelian 
topological group and the free Boolean topological group on $X$. Free Boolean 
topological groups and free Abelian topological groups share many properties. 
For example, $\ind A(X)=0$ and $\ind B(X)=0$ provided that $\dim X=0$. These 
facts were proved in \cite[Theorem~2.1]{Si5} and \cite[Theorem~8]{Sip}, 
respectively. It is also known that the free Abelian topological group $A(X)$ 
on a Dieudonn\'e complete space $X$ is Weil (equivalently, Ra\u{\i}kov) 
complete (see \cite[Theorem~4]{Tk83} or \cite[Theorem~7.9.6]{AT}). In 
\cite[Problem~6]{Sip} O.~Sipacheva asked whether the free Boolean topological 
group on a Dieudonn\'e complete space is Weil complete. Our main objective is 
to give a positive answer to Sipacheva's question.

The topology of a topological group is determined by the family 
of left-invariant continuous pseudometrics on the group 
\cite[Theorem~8.2]{HR}. Let $X$ be a space and $\mathcal{P}_X$
the family of continuous pseudometrics on $X$. In the case of 
the free Abelian topological group $A(X)$ on a space $X$, the 
topology of $A(X)$ is determined by the family of Graev 
extensions of pseudometrics in $\mathcal{P}_X$ to $A(X)$ (see
\cite{Marx} or \cite{Nick}). A similar result, Theorem~\ref{Th_Ab_desc},
is valid for the free Boolean group $B(X)$ on $X$. But the construction of the 
Graev extension of pseudometrics from $X$ to $B(X)$ and the properties of these 
extensions are quite different from those in the case of $A(X)$ (see 
\cite[p.~498]{Sip} regarding this issue). The Graev extension is our main 
technical tool in this article. Therefore, we present a (relatively) detailed 
exposition of all the necessary material regarding the Graev extension of 
pseudometrics from $X$ to $B(X)$ and then apply it to prove 
Theorem~\ref{Th:Main} that the group $B(X)$ is Weil complete if and only if the 
space $X$ is Dieudonn\'e complete.

\subsection{Notation and terminology}
We remind that a \emph{Boolean group} $B$ is a group in which all elements 
except the neutral one are of order 2. It is well known and easy to see 
that any Boolean group is Abelian; moreover, any such group is a vector 
space over the two-element field $\mathbb F_2$ and therefore has a basis. Thus, 
any abstract (without topology) Boolean group is free and can be identified 
with the set of all finite subsets of its basis under the operation of 
symmetric difference. Given a non-empty set $X$ with a distinguished point 
$e\in X$, we denote by $B_a(X)$ the \emph{abstract free Boolean group on $X$ in 
the sense of Graev}, i.e., the Boolean group with basis $X\setminus\{e\}$ and 
zero identified with the element $e$ of $X$. For a Tychonoff space $X$ (with a 
distinguished point $e$), the \emph{free Boolean topological group on $X$ in 
the sense or Graev}, or simply the \emph{free Boolean topological group on 
$X$}, is denoted by $B(X)$. It is defined by the following conditions: 
\begin{enumerate} 
\item[(i)]  
$B(X)$ is a topological Boolean group, $X$ is a subspace of $B(X)$, and $B(X)$ 
is generated by $X$; 
\item[(ii)] 
every continuous mapping $f\colon X\to H$ to a Boolean topological group $H$ 
satisfying $f(e)=e_H$ extends to a continuous homomorphism $\widehat{f}\colon 
B(X)\to H$. 
\end{enumerate} 
As an abstract group, $B(X)$ coincides with $B_a(X)$; in particular, the set 
$X\sm\{e\}$ is linearly independent in~$B(X)$. A standard argument (see 
\cite[Section~2]{Gra}) shows that $B(X)$ does not depend on the choice of the 
distinguished point $e\in X$; in other words, the groups $B(X)$ 
corresponding to all possible choices of $e$ are topologically isomorphic. 

The more familiar \emph{free Boolean topological group $B_M(X)$ on $X$ in the 
sense of Markov} is defined by the similar conditions  
\begin{enumerate} 
\item[(i$^\prime$)] 
$B_M(X)$ is a topological Boolean group, $X$ is a subspace of $B_M(X)$, and 
$B_M(X)$ is generated by $X$;
\item[(ii$^\prime$)] 
every continuous mapping from $X$ to a 
Boolean topological group $H$ extends to a continuous homomorphism 
$B_M(X)\to H$.
\end{enumerate}

Let $x_0\notin X$, and let $X_\ast$ be the topological sum $X\oplus\{x_0\}$, 
i.e., the space $X\cup \{x_0\}$ containing  $X$ as a clopen subspace. 
An easy verification similar to the one in \cite[Section~2]{Gra} shows that the 
groups $B_M(X)$ and $B(X_\ast)$ are topologically isomorphic. Thus, studying 
free Boolean topological groups in the sense of Markov reduces to studying free 
Boolean topological groups.

Every topological group $G$ with identity element $1$ and topology $\tau$ 
carries three natural group uniformities, the left uniformity $\mathscr V^l_G$ 
with base 
$$
\{\{(g,h)\in G\times G: h\in gV\}: \text{$V$ is a neighborhood of $1$}\}, 
$$
the right uniformity $\mathscr V^r_G$ with base 
$$
\{\{(g,h)\in G\times G: h\in Vg\}: \text{$V$ is a neighborhood of $1$}\}, 
$$ 
and the two-sided uniformity $\mathscr V_G$ with base 
$$ 
\{\{(g,h)\in G\times G: h\in gV\cap Vg\}: \text{$V$ is a neighborhood of 
$1$}\}. 
$$
It is well known that each of these uniformities induces the topology $\tau$ 
on~$G$ \cite[Example~8.1.17]{Eng}.

Recall that a filter $\mathscr F$ on a uniform space $X$ with uniformity 
$\mathscr U$ is said to be \emph{Cauchy} with respect to $\mathscr U$ if each 
$V\in \mathscr U$ contains $F\times F$ for some $F\in \mathscr F$. A filter 
$\mathscr F$ on $X$ \emph{converges} to a point $x\in X$ if $x\in \bigcap 
\{\overline F: F\in \mathscr F\}$ (the closures are in the topology generated 
by $\mathscr U$). 

A topological group $G$ is said to be  \emph{Ra\u{\i}kov 
complete} (\emph{Weil complete}) if it is complete with respect to $\mathscr 
V_G$ (with respect to $\mathscr V^l_G$). Basic information about Ra\u{\i}kov 
and Weil complete groups can be found in \cite[Section~3.6]{AT} (see also 
\cite{Raikov}). In topological Abelian groups, the uniformities $\mathscr 
V^l_G$ and $\mathscr V_G$ coincide, so that  Weil completeness and Ra\u{\i}kov 
completeness coincide as well. Note that a filter $\mathscr F$ on an Abelian 
group $G$ is Cauchy if and only if, for any neighborhood $U$ of zero in $G$, 
there exists an $F\in \mathscr F$ such that $F-F\subset U$. 

In what follows, given a space $X$, we denote the left (or, equivalently, 
two-sided) group uniformity $\mathscr V_{B(X)}$ by $\mathscr V$ and call it 
simply the \emph{group uniformity} of~$B(X)$. 

A space $X$ is \emph{Dieudonn\'e complete} if there exists a complete 
uniformity on $X$ compatible with the topology of $X$, which is the case if 
and only if the universal uniformity of $X$ is complete. Equivalently, $X$ is 
Dieudonn\'e complete if it is homeomorphic to a closed subspace of the 
Tychonoff product of a family of metrizable spaces. These and some other facts 
concerning Dieudonn\'e completeness can be found in \cite[Section~8.5.13]{Eng}.

We use the notation $\mathbb N$ for the set of positive integers and $\mathbb 
N_0$ for the set of non-negative integers. Given a subset $A$ of a 
topological space, by $\overline A$ we denote the closure of $A$ in this space.

%%%%%%%%%%%%%%%%%%%%%%%%%%%%%%%%%%%%%%%%%%%%
%%%%%%%%%%%%%%%%%%%%%%%%%%%%%%%%%%%%%%%%%%%%
\section{Completeness of $B(X)$}\label{Sec:2}

The main result of this paper is the following theorem.
 
\begin{theorem}\label{Th:Main}
The free Boolean topological group $B(X)$ on a Tychonoff 
space $X$ is Weil complete if and only if $X$ is Dieudonn\'e 
complete.
\end{theorem}

Our proof of Theorem~\ref{Th:Main} is based on 
several statements given below. Some of them are almost exact duplicates of 
earlier results for free Abelian topological groups (see 
\cite[Section~7.2]{AT}). To make our exposition self-contained, we give the 
proofs of all auxiliary results, occasionally omitting some standard 
components.

Recall that a pseudometric $\varrho$ on an Abelian group $G$ is said to be 
\emph{invariant} if $\varrho(x, y)= \varrho(x+z, y+z)$ for any $x, y, z\in G$. 

In \cite[p.~498]{Sip}, a procedure for extending continuous pseudometrics on 
a space $X$ to invariant continuous pseudometrics on $B(X)$ was described 
without proof (this is an adaptation of Graev's procedure for extending 
pseudometrics to free topological groups). We provide some details along the 
lines of~\cite{Sip}. 

Let $\varrho$ be a pseudometric on a non-empty set $X$. To begin, we extend 
$\varrho$ to $B_a(X)$. Given an $h\in B_a(X)\setminus \{e\}$, we consider all 
possible representations of $h$ in the form 
\begin{equation} 
\label{eq:r1} 
h= (x_1+y_1)+\cdots+(x_n+y_n), 
\end{equation} 
where $n\in \mathbb N$ and $x_i,y_i\in X$ for all $i\leq n$. To every 
representation $(\ref{eq:r1})$ of $h$ we assign the non-negative number 
\begin{equation}
\label{eq:sum2}
\sum_{i=1}^n \varrho(x_i,y_i). 
\end{equation} 
Let $N_\varrho(h)$ be the infimum of the sums $(\ref{eq:sum2})$ corresponding 
to all representations of $h$ in the form \eqref{eq:r1}. We have 
$N_\varrho(g+h)\leq N_\varrho(g) + N_\varrho(h)$, because the sum of 
representations of $g,h\in B_a(X)$ is a representation of $g+h$. It is also 
clear that $N_\varrho(e)=0$ and $N_\varrho(g)\geq 0$ for each $g\in B_a(X)$. 
Since every element $g\in B_a(X)$ satisfies $-g=g$, we see that $N_\varrho$ is 
a \emph{prenorm} on $B_a(X)$ in the sense of \cite[Section~3.3]{AT}. 

We set $\widehat{\varrho}(g,h)= N_\varrho(g+h)$ for $g,h\in B_a(X)$. Clearly, 
$\widehat \varrho$ is an invariant pseudometric on $B_a(X)$. 
Informally, the following lemma says that the distance $\widehat{\varrho}(g,h)$ 
between two elements $g=x_1+ \cdots+x_n$ and $h=y_1+\cdots+y_k$ of the group 
$B(X)$ depends solely on the restriction of the 
pseudometric $\varrho$ 
to the finite 
subset $\{e\}\cup \{x_1,\ldots,x_n\} 
\cup\{y_1,\ldots,y_k\}$ of~$X$.

\begin{lemma}\label{Le:clearer}
Suppose that $\varrho$ is a pseudometric on a set $X$ and 
$h=x_1+\cdots+x_n$ is an element of $B_a(X)\sm\{e\}$ 
with pairwise distinct $x_1,\ldots,x_n\in X\sm\{e\}$. Then 
there exists a representation
\begin{eqnarray*}
h=(u_1+v_1)+\cdots+(u_k+v_k) 
\end{eqnarray*}
such that $u_1,v_1,\ldots,u_k,v_k$ are  pairwise distinct and 
$$
\widehat{\varrho}(h,e)=\sum_{i=1}^k \varrho(u_i,v_i). 
$$
Moreover, if $n$ is even, then $2k=n$ and $\{u_1,v_1,\ldots,u_k,v_k\}=
\{x_1,\ldots,\allowbreak x_n\}$, and if $n$ is odd, then $2k=n+1$, 
$\{u_1,v_1,\ldots,u_k\}=\{x_1,\ldots, x_n\}$, and $v_k=e$. 

In particular, the quantity $N_\varrho(h)$, which is 
the infimum of the sums $(\ref{eq:sum2})$ over all possible representations 
of $h$ of the form \eqref{eq:r1}, is in fact the minimum of those sums. 
\end{lemma}

\begin{proof}
Let 
\begin{equation}
\label{eq:r3}
h = (z_1+t_1) + \cdots + (z_p+t_p)
\end{equation}
be a representation of $h$ with $z_i,t_i\in X$ for each $i\leq p$.
First, we claim that certain cancellations in $(\ref{eq:r3})$ yield another 
representation 
\begin{equation}
\label{eq:r6}
h = (a_1+b_1) + \cdots + (a_q+b_q)
\end{equation}
such that $q\leq p$, $\{a_1,b_1,\ldots,a_q,b_q\}\sub 
\{z_1,t_1,\ldots,z_p,t_p\}$, the elements $a_1,b_1,\ldots,a_q,b_q$
are pairwise distinct, and 
\begin{equation}
\label{eq:r9}
\sum_{i=1}^q \varrho(a_i,b_i)\leq 
\sum_{i=1}^p \varrho(z_i,t_i).
\end{equation}
Clearly, we may assume that $z_i\neq t_i$ for each $i\leq p$; otherwise,
we delete the summand $z_i+t_i$ from the representation $(\ref{eq:r3})$
of $h$, affecting neither $h$ nor the sum on the right-hand side of 
$(\ref{eq:r9})$. Suppose that $i\ne j$ and the summands $(z_i+t_i)$
and $(z_j+t_j)$ in $(\ref{eq:r3})$ contain equal elements. We may assume 
without loss of generality that $z_i=z_j$, because the group $B_a(X)$ is 
Abelian. We replace the two summands $(z_i+t_i)$ and $(z_j+t_j)$ in 
$(\ref{eq:r3})$ with the single summand $(t_i+t_j)$, thus obtaining a shorter 
representation of $h$. In view of the triangle inequality we have 
$\varrho(t_i,t_j)\leq \varrho(t_i,z_i) + \varrho(z_j,t_j)$ (here $z_i=z_j$), 
so that the sum of distances for the new representation of $h$ does not 
increase. Proceeding, we obtain a reduced representation of $h$ of the form 
$(\ref{eq:r6})$ satisfying condition $(\ref{eq:r9})$. In particular, all 
elements $a_1,b_1,\ldots,a_q,b_q$ of this representation are pairwise 
distinct and $\{a_1,b_1,\ldots,a_q,b_q\}\sub \{z_1,t_1,\ldots,z_p,t_p\}$, 
because the transformations of the given representation of $h$ performed above 
introduce no new elements of~$X$.

Note that $h=x_1+\dots+x_n$ is an expansion of $h$ in the basis $X\setminus 
\{e\}$. Since the elements $a_1,b_1,\ldots,a_q,b_q$ are 
pairwise distinct, it follows that at most one of them equals $e$ and 
the remaining ones belong to $X\setminus \{e\}$. 
Therefore, $\{a_1,b_1,\ldots,a_q,b_q\}\subset \{x_1,\dots, x_n\}\cup\{e\}$.  
Clearly, $h$ has only finitely many representations of the form \eqref{eq:r6} 
satisfying this condition. Since inequality \eqref{eq:r9} holds for all such 
representations, it follows from the arbitrariness of the representation 
\eqref{eq:r3}  that one of them, say, $h = (u_1+v_1)+\dots+ (u_k+v_k)$, 
satisfies
$$ 
\widehat{\varrho}(h,e) =N_\varrho(h)=\sum_{i=1}^q \varrho(a_i,b_i). 
$$

Since $u_1,v_1,\ldots,u_k,v_k$ are pairwise distinct
and $\{u_1,v_1,\ldots,u_k,v_k\}\sub\{x_1,\ldots,x_n,e\}$, we have $n\leq 2k\leq 
n+1$. Therefore, $n=2k$ if $n$ is even and $n+1=2k$ if $n$ is odd. In the 
latter case, exactly one of the elements 
$u_1,v_1,\ldots,u_k,v_k$ equals $e$, and we can assume that this is $v_k$, 
because $B_a(X)$ is Abelian. This completes the proof of the lemma. 
\end{proof}

\begin{theorem}\label{Th_Gra}
For any pseudometric $\varrho$ on a non-empty set $X$, 
the invariant pseudometric $\widehat\varrho$ on the abstract group $B_a(X)$ is 
an extension of $\varrho$. Moreover, if $X$ is a Tychonoff space and $\varrho$ 
is a continuous pseudometric on $X$, then $\widehat{\varrho}$ is a continuous 
pseudometric on~$B(X)$. 
\end{theorem}

\begin{proof}
According to Lemma~\ref{Le:clearer}, we have $N_\varrho(x+y)=\varrho(x,y)$
for all $x,y\in X$. Therefore, the pseudometric 
$\widehat{\varrho}$ on $B_a(X)$, which is defined by
\begin{eqnarray*}
\widehat{\varrho}(g,h) = N_\varrho(g+h),  
\end{eqnarray*}
extends $\varrho$, that is, $\widehat{\varrho}(x,y)
=\varrho(x,y)$ for all $x,y\in X$. This proves
the first part of the theorem. 

Suppose that $\varrho$ is a continuous pseudometric on $X$.
Since the group $B_a(X)$ is Abelian, the prenorm $N_\varrho$
induces a (possibly non-Hausdorff) group topology $\tau_\varrho$  on $B_a(X)$ 
whose local base at the zero element $e$ consists of the sets 
$$ 
O_\varepsilon = \{g\in B_a(X): N_\varrho(g)<\varepsilon\},
$$
where $\varepsilon>0$. It follows from the continuity of $\varrho$ that the 
restriction of $\tau_\varrho$ to $X$ is coarser than the topology of~$X$. 

Let us denote the topology of $B(X)$ by $\tau$. It follows from the definition 
of the free Boolean topological group $B(X)$ that $\tau$ is the finest 
group topology among those agreeing with the topology of $X$. Therefore, 
$\tau_\varrho$ is coarser than $\tau$, and hence $\widehat{\varrho}$ is a 
continuous pseudometric on the topological group~$B(X)$. 
\end{proof}

We refer to $\widehat \varrho$ as the \emph{Graev extension} of the 
pseudometric $\varrho$. 

\begin{lemma}\label{L_unif}
The restriction $\mathcal{V}_X=\mathcal{V}\res X$ of the 
group uniformity $\mathcal{V}$ of $B(X)$ to the subspace 
$X\sub B(X)$ coincides with the universal uniformity $\mathcal{U}_X$ 
of~$X$.
\end{lemma}

\begin{proof}
The uniformity $\mathcal{V}$ induces the topology of $B(X)$ on $B(X)$ and hence 
$\mathcal{V}_X$ induces the topology of $X$ on  $X\sub B(X)$. Therefore, 
$\mathcal{V}_X$  is coarser than $\mathcal{U}_X$. Conversely, take an arbitrary 
entourage $U\in\mathcal{U}_X$ of the diagonal in $X^2$. By 
\cite[Theorem~8.1.10]{Eng}, there exists a continuous pseudometric $\varrho$ on 
$X$ such that 
$$ 
U_\varrho=\{(x,y)\in X\times X: \varrho(x,y)<1\} \sub U. 
$$ 
Let $\widehat{\varrho}$ be the Graev extension of $\varrho$ to $B(X)$. Then 
$$ 
V_\varrho = \{(g,h)\in B(X)\times B(X): \widehat{\varrho}(g,h)<1\} 
$$ 
is an element of $\mathcal{V}$, because $\widehat{\varrho}$ is a continuous 
invariant pseudometric on $B(X)$. Clearly, 
$U_\varrho=V_\varrho\cap (X\times X)$. Hence $\mathcal{V}_X$ is finer than 
$\mathcal{U}_X$. Therefore, $\mathcal{V}_X=\mathcal{U}_X$. 
\end{proof}

\begin{lemma}\label{Le:Max}
The Graev extension $\widehat{\varrho}$ of a pseudometric 
$\varrho$ on $X$ to $B_a(X)$ is a maximal invariant pseudometric 
on $B_a(X)$ extending $\varrho$. In other words, if $d$ is an 
invariant pseudometric on $B_a(X)$ extending $\varrho$, then 
$d(g,h)\leq \widehat{\varrho}(g,h)$ for all $g,h\in B_a(X)$.
\end{lemma}

\begin{proof}
Let $d$ be an arbitrary invariant pseudometric on $B_a(X)$ 
extending $\varrho$, and let $u$ and $v$ be any elements of $B_a(X)$. Since
$d$ and $\widehat{\varrho}$ are invariant, we have $d(u,v)=d(h,e)$ and 
$\widehat{\varrho}(u,v)= \widehat{\varrho}(h,e)$ for $h=u+v$. Therefore, 
it suffices to verify that $d(h,e)\leq \widehat{\varrho}(h,e)$ for each $h\in 
B(X)$. Note that $d(x,y)=\varrho(x,y)$ for all $x,y\in X$, so the invariance 
of $d$ and $\widehat{\varrho}$ implies 
\begin{equation} 
\label{eq:ps1}
d(x+y,e)=d(x,y) = 
\varrho(x,y) = \widehat{\varrho}(x,y) = \widehat{\varrho}(x+y,e). 
\end{equation}

If $h=e$, then $d(h,e)=0=\widehat{\varrho}(h,e)$. So we
assume that $h\neq e$. Let $h=x_1+\cdots+x_n$, where 
$x_1,\ldots,x_n$ are
pairwise distinct elements of $X\setminus\{e\}$. If $n=1$, then $h=x_1\in X$, 
so $d(x_1,e)=\varrho(x_1,e)=\widehat{\varrho}(x_1,e)$. Suppose that 
$n\geq 2$. By Lemma~\ref{Le:clearer}, there exists a representation 
$h=(a_1+b_1)+\cdots+(a_p+b_p)$ with pairwise distinct 
$a_1,b_1,\ldots,a_p,b_p\in X$ such that 
$$ 
\widehat{\varrho}(h,e) = \sum_{i=1}^p \varrho(a_i,b_i). 
$$ 
For every $i\leq p$, by virtue of $(\ref{eq:ps1})$ 
we have $d(a_i+b_i,e)=d(a_i,b_i)=\varrho(a_i,b_i)$. Since $d$ is an 
invariant pseudometric on $B(X)$ and $h=\sum_{i=1}^p (a_i+b_i)$, it follows 
that 
$$ 
d(h,e)\leq \sum_{i=1}^p d(a_i+b_i,e) = \sum_{i=1}^p 
\varrho(a_i,b_i) = \widehat{\varrho}(h,e). 
$$ 
This proves the lemma. 
\end{proof}

\begin{theorem}\label{Th_Ab_desc}
Let $X$ be a Tychonoff space and $\mathcal{P}_X$ be the family 
of all continuous pseudometrics on $X$. Then the sets
\[
V_\varrho=\{g\in B(X): \widehat\varrho(g,e)<1\}, 
\]
where $\varrho\in\mathcal{P}_X$, form a local base at
the zero element $e$ of~$B(X)$.
\end{theorem}

\begin{proof}
Consider an arbitrary open neighborhood $U$ of $e$ in $B(X)$. 
According to \cite[Theorem~8.2]{HR}, there exists a continuous 
invariant pseudometric $d$ on $B(X)$ such that 
$$
V_d=\{g\in B(X): d(g,e)<1\}\sub U.
$$
We denote by $\varrho$ the restriction of $d$ to the subset $X$
of $B(X)$ and let $\widehat{\varrho}$ be the Graev extension
of $\varrho$ to $B(X)$. Lemma~\ref{Le:Max} implies 
$d\leq \widehat{\varrho}$, so that 
$$
\{g\in B(X): \widehat{\varrho}(g,e)<1\}\sub V_d\sub U.
$$
By Theorem~\ref{Th_Gra}, the pseudometric $\widehat{\varrho}$
is continuous, and hence $V_\varrho$ is open in $B(X)$. This completes the 
proof of the theorem. 
\end{proof}

Given a subspace $Y$ of $X$ containing the point $e$, let $B(Y|X)$ denote the 
topological subgroup of $B(X)$ generated by $Y$. The topology of $B(Y|X)$ may 
be strictly coarser than that of the free Boolean topological group $B(Y)$. 
Theorem~\ref{Th_Ab_desc}, together with Lemmas~\ref{Le:clearer} and~\ref{L_unif}, 
implies the following important statement. 

\begin{corollary}[{see also \cite[Theorem~4]{Sip}}]
Let $X$ be a space with a distinguished point $e$, and let $Y$ be its subspace 
containing $e$. The topological subgroup $B(Y|X)$ of the free Boolean 
topological group $B(X)$ generated by $Y$ is the free topological group $B(Y)$ 
if and only if every continuous pseudometric on $Y$ can be extended to 
a continuous pseudometric on $X$. 
\end{corollary}

\begin{proof}
As is known, the group uniformity of any subgroup $H$ of a topological 
group $G$ coincides with the restriction of the group uniformity of $G$ 
to $H$. Therefore,  if the topology of $B(Y|X)$ coincides with that of $B(Y)$, 
then by Lemma~\ref{L_unif} the uniform space $(Y, \mathscr U_Y)$ must be a 
uniform subspace of $(X, \mathscr U_X)$, which means that any continuous 
pseudometric on $Y$ can be extended to a continuous pseudometric on $X$ (in 
this case, $Y$ is said to be $P$-embedded in $X$ \cite{Shapiro}). Conversely, 
if any continuous pseudometric on $Y$ is the restriction to $Y$ of a continuous 
pseudometric on $X$, then by Lemma~\ref{Le:clearer} the Graev extension of 
any continuous pseudometric on $Y$ to $B(Y)$ is the restriction to $B(Y|X)$ of 
the Graev extension of a continuous pseudometric on $X$. According to 
Theorem~\ref{Th_Ab_desc}, the topology of $B(Y|X)$ is finer than that of 
$B(Y)$. But it is also coarser than the latter, because it follows from the 
definition of the free Boolean topological group $B(Y)$ that the identity 
homomorphism $B(Y)\to B(Y|X)$ extending the identity mapping $Y\to Y$ must be 
continuous. 
\end{proof}

We set\label{B_n(X)}  
$$
B_0(X)=\{e\} 
$$
and 
$$
B_n(X)= \{x_1+\dots+x_k\in B(X): k\le n,\ x_1,\dots, x_k\in X\}.
$$ 
for $n\ge 1$.
We define mappings $i_n\colon X^n\to B(X)$ 
by $i_n(x_1,\ldots,x_n)=x_1+\cdots+x_n$ for $n\in \mathbb N$ and $x_1,\ldots,x_n
\in X$. Clearly, each $i_n$ is continuous and $i_n(X^n)=B_n(X)$. 

\begin{corollary}[{see also \cite[p.~501]{Sip}}]
\label{Th_clo}
For a Tychonoff space $X$, the set $B_n(X)$ 
is closed in $B(X)$ for each $n\in \mathbb N_0$. In particular, 
$X$ is closed in $B(X)$. 
\end{corollary}
 
\begin{proof}
Let $Y=bX$ be a Hausdorff compactification of $X$.
It follows from the definition of free Boolean topological groups
that the identity embedding $X\to Y\sub B(Y)$ extends to a
continuous homomorphism $f\colon B(X)\to B(Y)$. Clearly,
$f$ is one-to-one, i.e., $f$ is a continuous isomorphism (not
necessarily a topological embedding). 

Since $Y$ and $Y^n$ are compact, the set $i_n(Y^n)=B_n(Y)$ 
is compact and closed in $B(Y)$. It follows from 
$B_n(X) = f^{-1}(B_n(Y))$ that $B_n(X)$ is closed in $B(X)$ 
for each $n\in \mathbb N$.
 \end{proof} 
 
The following simple fact borrowed from \cite[Lemma~7.9.3]{AT} 
plays a crucial role in the proofs of Lemma~\ref{Le:modif} and 
Theorem~\ref{Th:Main}.

\begin{lemma}\label{Le:add1}
Let $K$ be a subset of an Abelian topological group $G$ with zero $0$
and group uniformity $\mathcal{V}_G$. If the uniform 
space $(K,\mathcal{V}_G\res K)$ is complete, then every
Cauchy filter $\mathscr F$ on $G$ satisfies one of the
following two conditions:
\begin{enumerate}
\item[\textup{(a)}] 
$\mathscr F$ converges to some point $x\in K$;
\item[\textup{(b)}] 
there exists an $F\in\mathscr F$ and a neighborhood $V$
of $0$ in $G$ such that $K\cap (F+V)=\emp$.
\end{enumerate}
\end{lemma}

\begin{proof}
Suppose that (a) does not hold and consider the family
\[
\gamma=\{F+V: F\in\mathscr F,\ V\in\mathcal{N}(0)\},
\]
where $\mathcal{N}(0)$ is the set of all open neighborhoods of $0$ in $G$. It 
is easy to see that $\gamma$ is a base of some Cauchy filter 
$\widetilde{\mathscr F}$ on $(G,\mathcal{V}_G)$. Indeed, let 
$\widetilde{\mathscr F}$ be the family of all subsets of $G$ containing at 
least one element of $\gamma$ and take an arbitrary $U\in\mathcal{N}(0)$. 
Choose a symmetric neighborhood $V\in\mathcal{N}(0)$ so that $V+V+V\sub U$. 
There exists an $F\in\mathscr F$ such that $F-F\sub V$. Hence 
$F+V\in\widetilde{\mathscr F}$ and 
$$
(F+V)-(F+V)=(F-F)+(V-V)\sub V+V+V\sub U. 
$$ 
It follows that $\widetilde{\mathscr F}$ is a Cauchy filter on 
$(G,\mathcal{V})$.

If the intersection $K\cap P$ is non-empty for each $P\in\widetilde{\mathscr 
F}$, then the family $\{K\cap P: P\in\widetilde{\mathscr F}\}$ forms a Cauchy 
filter on the uniform space $(K,\mathcal{V}_G\res K)$, which converges to a 
point $x\in K$, because this space is complete. Therefore, both filters 
$\widetilde{\mathscr F}$ and $\mathscr F$ converge to $x$, which contradicts 
our assumption about $\mathscr F$. Thus, there exists an element 
$P\in\widetilde{\mathscr F}$ of the form $P=F+V$, where $F\in\mathscr F$ 
and $V\in\mathcal{N}(0)$, such that $P\cap K=\emp$. So (b) holds. 
\end{proof}

Recall that on p.~\pageref{B_n(X)} we defined the ``small'' subspaces 
$B_n(X)$ of $B(X)$. Let $\mathscr V_n$ denote the restriction of the 
group uniformity $\mathscr V$ of $B(X)$ to $B_n(X)$. The following lemma says 
that if $X$ is Dieudonn\'e complete, then the subspaces $B_n(X)$ are complete 
with respect to $\mathscr V_n$. In its proof we use the notation 
$$ 
\overline{\mathscr F}= \{\overline F: F\in \mathscr F\}. 
$$ 

\begin{lemma}\label{Le:modif}
Let $X$ be a Dieudonn\'e complete space. Then the subspace
$B_n(X)$ of $B(X)$ with the uniformity inherited from
$(B(X),\mathcal{V})$ is complete for each $n\in \mathbb N_0$.
\end{lemma}

\begin{proof}
We prove the lemma by induction on $n$. For $n=0$, there is nothing to 
prove. Suppose that $n>0$ and we have already proved the completeness of 
$(B_{n-1}(X), \mathscr U_{n-1})$. Consider any Cauchy filter $\mathscr 
F$ on $(B_n(X), \mathscr U_n)$; note that $\mathscr F$ is a base of a 
Cauchy filter on $B(X)$. We must show that $\bigcap \overline{\mathscr F}\ne 
\varnothing$. 

If $\bigcap \overline{\mathscr F}\cap B_{n-1}(X)\ne \varnothing$, there is 
nothing to prove. Suppose that $\bigcap \overline{\mathscr F}\cap B_{n-1}(X)= 
\varnothing$. Then by the induction hypothesis and Lemma~\ref{Le:add1} there exists 
an $F_0\in \mathscr F$ and a neighborhood $U_0$ of zero in $B(X)$ for which 
$(F_0+ U_0)\cap B_{n-1}(X) = \varnothing$. Let $\varrho_0$ be a continuous 
pseudometric on $X$ for which $V_{\varrho_0}\subset U_0$ (it exists by 
Theorem~\ref{Th_Ab_desc}). Then 
\begin{equation} 
\label{eq1} 
F_0 + U_{d_0} \cap B_{n-1}(X) = \varnothing. 
\end{equation} 
Note that every element of $F_0$ can be written as $x_1+\dots +x_n$, where the 
$x_i$ are pairwise distinct elements of $X\setminus \{e\}$, because $F_0\subset 
B_n(X)\setminus B_{n-1}(X)$. For the same reason, if $x_1, \dots, x_n\in X$ and 
$x_1+\dots +x_n\in F_0$, then the $x_i$ are pairwise distinct elements of 
$X\setminus \{e\}$. In what follows, by the letter $x$ (with indices and 
possibly primes) we always denote elements of $X\setminus \{e\}$.  
 
We claim that 
$$ 
\begin{gathered} 
d(x_i, x_j)\ge 1\\
\text{for any $g = x_1+\dots +x_n\in F_0$ and any $i,j \le n$, $i\ne j$.}
\end{gathered}\eqno{\text{(A)}} 
$$ 
Indeed, otherwise, $x_i+x_j\in V_{\varrho_0}$ and 
$\overline x + x_i+x_j \in B_{n-2}(X)\subset B_{n-1}(X)$, which 
contradicts~\eqref{eq1}. 

We set $\varrho^*=3\varrho_0$ and choose $F^*\in \mathscr F$ so that 
\begin{enumerate}
\item[(i)] 
$F^*\subset F_{\varrho_0}$ and 
\item[(ii)] 
$F^* + F^*\subset V_{\varrho^*}$. 
\end{enumerate}
It follows from (i) and (A) that 
$$
\begin{gathered} 
\varrho^*(x_i, x_j)\ge 3\\
\text{for any $x_1+ \dots + x_n\in F^*$ and any $i, j\le n$, $i\ne j$.} 
\end{gathered}
\eqno{\text{(B)}} 
$$
Combining (ii) and (B), we see that, given any $x'_1+ \dots + x'_n,\, 
x''_1+ \dots + x''_n\in F^*$, each $x'_i$ either is cancelled in $x'_1+ \dots + 
x'_n+ x''_1+ \dots + x''_n$ (in which case $x'_i=x''_j$ for some $j$) or is at 
a distance $\varrho^*$ less than 1 from some $x''_j$, i.e., 
$$
\begin{gathered} 
\text{given any $x'_1+ \dots + x'_n, \,x''_1+ \dots + x''_n\in F^*$,}\\ 
\text{for 
each $i\le n$, there exists a $j\le n$ such that} \\
\varrho_0(x'_i, x''_j)<1/3. 
\end{gathered}
\eqno{\text{(C)}}
$$

Fix any $x^*_1+\dots +x^*_n\in F^*$. Let $\varepsilon$ be a positive number 
satisfying the conditions $\varepsilon \le 1/3$ and $\varepsilon < 
\varrho_0(x_i^*, e)$ for all $i\le n$. We set 
$$
O_i=\{x\in X: \varrho_0(x^*_i, x)\le \varepsilon\}.
$$ 
Note that in view of (B) we have $O_i\cap O_j=\varnothing$ for distinct $i, 
j\le n$. Therefore, it follows from (C) that, for any $x_1+\dots+x_n\in 
F^*$, there exists a unique permutation $\pi$ of $\{1,...,n\}$ such that 
$x_{\pi(i)}\in O_i$ for every $i\le n$. Thus, $F^*\subset O_1+\dots 
+O_n$. 

Since $O_i\subset X\setminus \{e\}$ for $i\le n$, it follows that, given 
any $g = x_1+\dots +x_n\in O_1+\dots +O_n$, every $O_i$ contains precisely 
one of the elements $x_1,\dots,x_n$. In other words, these elements can be 
renumbered as $x'_1, \dots, x'_n$ so that $x'_i \in O_i$ for each $i\le n$, and 
this renumbering is determined uniquely. In the rest of the proof of this lemma 
we assume that the elements of each $g \in O_1+\dots +O_n$ are 
numbered in this way; that is, considering $x_1+\dots +x_n\in O_1+\dots +O_n$, 
we always assume that $x_i\in O_i$ for $i\le n$. 

The restriction $\tilde \imath_n$ of the natural addition map $i_n\colon X^n\to 
B_n(X)$ (defined by $i_n(x_1, \dots, x_n) = x_1 + \dots + x_n$) to $O_1\times 
\dots \times O_n$ is continuous and bijective on $O_1\times \dots \times O_n$, 
and since each $O_i$ is closed in $X$, it follows that 
$\tilde \imath_n^{-1}(\overline F)$ is closed in $X^n$ for any $F\in \mathscr 
F$ such that $F\subset F^*$. 

Let $\varrho$ be any continuous pseudometric on $X$ satisfying the 
condition $\varrho\ge \varrho^*$, and let $F\in \mathscr F$ be such that 
$F\subset F^*$ and $F+F\subset V_\varrho$. Arguing as above and taking into 
account our convention on numbering, we see that 
$$
\begin{gathered}
d(x'_i, x''_i)<1\\
\text{whenever $x'_1+ \dots + x'_n,\, x''_1+ \dots + x''_n\in F$ and $i\le n$.}
\end{gathered}
$$ 
Thus, the sets $\tilde\imath^{-1}(F)$, where $F\in \mathscr F$, $F\subset 
F^*$, form a base of a Cauchy filter on the product $(X^n, \mathscr U_X^n)$. 
Since $(X^n, \mathscr U_X^n)$ is complete as a product of complete spaces 
(see, e.g., \cite[Theorem~8.3.9]{Eng}), it follows that there exists an 
$(\tilde x_1, \dots, \tilde x_n)\in \bigcap
\{\overline{\tilde\imath^{-1}(F)}: F\in \mathscr F,\ F\subset F^*\}$. Clearly, 
$\imath(\tilde x_1, \dots, \tilde x_n)\in \bigcap \overline{\mathscr F}$, 
because $\tilde\imath^{-1}(\overline F)$ is closed and hence  
$\overline{\tilde\imath^{-1}(F)} \subset \tilde\imath^{-1}(\overline F)$ for 
every $F\in\mathscr F$ such that $F\subset F^*$.
\end{proof}

The next lemma gives an alternative description of a neighborhood
base at the zero element of the group $B(X)$. 

\begin{lemma}\label{Le:sum}
Let $D=\{d_n: n\in\mathbb{N}\}$ be a sequence of continuous
pseudometrics on a Tychonoff space $X$. For every $n\in\mathbb{N}$,
let $W_n=\{x+y: d_n(x,y)<1\}$ be a subset of $B_2(X)$. Then the set
$$ 
W_D=\bigcup_{n=1}^\infty \big(W_1+\cdots+W_n\big)
$$
is an open neighborhood of $e$ in $B(X)$.
\end{lemma}

\begin{proof}
For every $n\in \mathbb N$, we set 
$p_n=\sum_{i=1}^{2^{n+1}}d_i$ and 
\begin{equation} 
\label{eq:alpha}
O_n= \{(x,y)\in X\times X: p_n(x,y)<1\}. 
\end{equation} 
According to \cite[Theorem~8.1.10]{Eng}, there exists a continuous pseudometric  
$\varrho$ on $X$ such that 
\begin{equation} 
\label{eq:beta}
\{(x,y)\in X\times X: \varrho(x,y)<2^{-n}\}\sub O_n 
\end{equation} 
for each $n$. Let us show that the open ball 
$$ 
O_\varrho = \{g\in B(X): \widehat{\varrho}(g,e)<1/2\} 
$$ 
centered at $e$ is contained in~$W_D$.

Let $g$ be an arbitrary element of $O_\varrho$. Then
$\widehat{\varrho}(g,e)<1/2$ and Lemma~\ref{Le:clearer} 
implies the existence of a representation
$$
g=(x_1+y_1)+\cdots+(x_m+y_m) 
$$
with $x_i,y_i\in X$ such that
\begin{equation}
\label{eq:B}
\widehat{\varrho}(g,e)=\sum_{i=1}^m \varrho(x_i,y_i). 
\end{equation}
For every $i\le m$ satisfying the condition $\varrho(x_i,y_i)\neq 0$, we 
choose  a $k_i\in \mathbb N$ so that $2^{-k_i-1}\leq 
\varrho(x_i,y_i)<2^{-k_i}$. For the remaining indices $i\le m$, using 
\eqref{eq:B}, we choose sufficiently large $k_i\in \mathbb N$ so that 
$\sum_{i=1}^m 2^{-k_i}<2 \cdot\widehat{\varrho}(g,e)<1$. Note that, for every 
$j\in \mathbb N$, the sum $\sum_{i=1}^m 2^{-k_i}$ contains fewer than 
$2^j$ summands equal to $2^{-j}$. Let $k_{i_1},\ldots,k_{i_r}$ be the list of 
all $k_i$ satisfying $k_i=j$; then $r<2^j$. From our choice of the numbers 
$k_i$ it follows that $\varrho(x_{i_s},y_{i_s})<2^{-j}$ for each $s$ with 
$1\leq s\leq r$. Hence $(\ref{eq:alpha})$ and $(\ref{eq:beta})$ imply that 
$d_m(x_{i_s},y_{i_s})\leq p_j(x_{i_s},y_{i_s})<1$ for all $s\leq r$ and 
$m\leq 2^{j+1}$ (both sides of this inequality depend on $j$). 
Therefore, 
\begin{multline*}
(x_{i_1}+y_{i_1})+\cdots+ 
(x_{i_r}+y_{i_r}) \in  W_{2^j+1}+ 
W_{2^{j}+2}+\cdots+W_{2^j+r}\\
 \sub  W_{2^{j}}+W_{2^{j}+1}+\cdots + W_{2^j+r}+
\cdots + W_{2^{j+1}-1}.
\end{multline*}
This inclusion holds for each $j\le k=\max\{k_1,\ldots,k_m\}$.
Thus, we have
$$
g=(x_1+y_1)+\cdots+(x_m+y_m)\in W_1+W_2+\cdots+W_{2^{k+1}}\sub W_D.
$$
We conclude that $W_D$ contains the open neighborhood $O_\varrho$ of $e$ 
in $B(X)$. 

To show that $W_D$ is open in $B(X)$, take an arbitrary element $g\in W_D$.
There exists an $n\in\mathbb N$ for which $g\in W_1+\cdots+W_n$.
Arguing as in the first part of the proof, we see that the set $W_{n+1}+
W_{n+2}+\cdots$ contains a neighborhood $U$ of $e$ in $B(X)$. The neighborhood 
$g+U$ of $g$ satisfies 
$$ 
g+U\sub(W_1+\cdots+W_n)+(W_{n+1}+W_{n+2}+\cdots)=W_D.
$$
Therefore, $W_D$ is open in $B(X)$.
\end{proof}

\begin{proof}[Proof of Theorem~\ref{Th:Main}]
Let $X$ be a Dieudonn\'e complete space. Suppose for a contradiction that there 
exists a Cauchy filter $\mathscr F$ on $(B(X), \mathscr U)$ for which $\bigcap 
\{\overline F: F\in \mathscr F\}=\varnothing$. Using Lemmas~\ref{Le:add1},  
\ref{Le:modif}, and~\ref{Le:sum}, we find a sequence of $F_n\in \mathscr F$, 
$n\in \mathbb N$, and a sequence of sequences $D_n=(d_n(k))_{k\in 
\mathbb N}$, $n\in \mathbb N$, of continuous pseudometrics on $X$ such 
that $F_n\cap (B_{3n}(X)+W_{D_n})=\varnothing$, $F_{n+1}\subset F_n$, 
$d_{n+1}(k)\ge d_{n}(k)$, and $d_k(n+1)\ge d_k(n)$ for 
all $n,k\in \mathbb N$. Let $D=(d_n(n))_{n\in \mathbb N}$. 

For each $n\in \mathbb N$, we have $F_n\cap 
(B_n(X)+W_D)=\varnothing$. Indeed, any $g\in W_D$ can be 
represented as $g=x_1+y_1+x_2+y_2+ \dots +x_k+y_k$, where $k \in \mathbb N$ 
and,  for each $i\le k$, $x_i, y_i\in X$ and $d_i(i)(x_i, y_i)<1$. If $k\le n$, 
then $B_n(X)+g\subset B_{3n}(X)$. Since $F_n\cap 
(B_{3n}(X)+W_{D_n})=\varnothing$ and $W_{D_n}$ is a neighborhood of zero, it 
follows that $F_n\cap B_{3n}(X)=\varnothing$ and hence $(B_n(X)+g)\cap F_n = 
\varnothing$. Suppose that $k>n$. For any $h\in B_n(X)$, we have 
$h+x_1+y_1+x_2+y_2+\dots +x_n+y_n\in B_{3n(X)}$, and for $i>n$, 
$d_i(i)\ge d_n(i-n)$, so that $x_{n+1}+y_{n+1}+\dots +x_k+ y_k\in 
W_{D_n}$. Therefore, $h+g\in B_{3n}(X)+W_{D_n}$. Thus, $B_n(X)+g\sub 
B_{3n}(X)+W_{D_n}$. It follows from the arbitrariness of $g\in W_D$ that  
$B_n(X)+W_D\sub B_{3n}(X)+W_{D_n}$. Since $F_n\cap 
(B_{3n}(X)+W_{D_n})=\varnothing$, we have $F_n\cap 
(B_n(X)+W_D)=\varnothing$. 

Take an $F\in \mathscr F$ such that $F+F\subset W_D$. Choose any $x^*_1+\dots+ 
x^*_k\in F$; we assume that $x^*_i\in X$ for $i\le k$. We have $x^*_1+\dots + 
x^*_k+F\subset W_D$ and hence $F\subset x^*_1+\dots x^*_k+W_D\subset 
B_k(X)+W_D$. This means that $F\cap F_k=\varnothing$, which contradicts 
$\mathscr F$ being a filter. 
\end{proof}

\end{document}